\documentclass[reqno,12pt]{amsart}
\numberwithin{equation}{section}

\theoremstyle{plain}
\newtheorem{thm}{Theorem}[section]
\newtheorem{cor}[thm]{Corollary} 
\newtheorem{lemma}[thm]{Lemma} 
\newtheorem{prop}[thm]{Proposition}

\newtheorem{hyp}[thm]{Hypothesis}

\theoremstyle{remark}

\newtheorem{remark}[thm]{Remark}

\theoremstyle{definition}

\newtheorem{defi}[thm]{Definition}

\usepackage{amscd,amssymb,comment,epic,eepic,euscript,graphics}

\newcommand{\la}{\lambda}
\newcommand{\al}{\alpha}
\newcommand{\De}{\Delta}

\newcommand\Rfr{{\mathfrak R}}
\newcommand\Mcal{{\mathcal{M}}} 
\newcommand\Reals{{\mathbb R}}
\newcommand\Complex{{\mathbb C}}
\newcommand\Nats{{\mathbb N}}
\newcommand{\BH}{\mathcal{B}(\mathcal{H})}
\newcommand{\Hcal}{\mathcal{H}}

\newcommand{\st}{\,:\,}

\renewcommand{\i}{\text{\rm i}}

\newcommand{\norm}[1]{\left\Vert#1\right\Vert}

\newcommand\ncs[1]{{\mathcal{L}_#1(\Mcal,\tau)}}
\newcommand{\dd}[3]{\De_{\la_1,\dots,#1}^{(#2)}\left(#3\right)}

\begin{document}

\title{Multiple operator integrals and spectral shift$^{*}$}

\author[Skripka]{Anna Skripka}
\address{Department of Mathematics, University of Central Florida, 4000 Central Florida
  Blvd., P.O.\ Box 161364, Orlando, FL 32816-1364, USA}
\email{skripka@math.ucf.edu}
\thanks{\footnotesize $^{*}$Research supported in part by NSF grant DMS-0900870}



\date{30 July, 2009}


\begin{abstract}
Multiple scalar integral representations for traces of operator derivatives are obtained and applied in the proof of existence of the higher order spectral shift functions.
\end{abstract}

\maketitle

\section{Introduction}

For a large class of admissible functions $f:\Reals\mapsto\Complex$, the operator derivatives
$\frac{d^j}{dx^j}f(H_0+xV)$, where $H_0$ and $V$ are self-adjoint operators on a separable Hilbert space $\Hcal$, exist and can be represented as multiple operator integrals \cite{Azamov0,PellerMult}. We explore properties of operator derivatives inside a semi-finite normal faithful trace $\tau$ given on a semi-finite von Neumann algebra $\Mcal$
acting on $\Hcal$.

For $H_0=H_0^*$  affiliated with $\Mcal$ and $V=V^*$ in the $\tau$-Hilbert-Schmidt class $\ncs{2}$ (that is, $V\in\Mcal$ and $\tau(|V|^2)<\infty$), we represent the traces of the derivatives $\tau\left[\frac{d^j}{dx^j}f(H_0+xV)\right]$ as multiple scalar integrals, and, subsequently, as a distribution on $f^{(j)}$, which is essentially a derivative of an $L^\infty$-function (see Theorem \ref{prop:two} and Corollary \ref{prop:two_cor}). We also obtain that the order of an operator derivative inside the trace can be decreased, which costs the increase of the order of a scalar derivative;  more precisely,
\[\tau\left[\frac{d^j}{dx^j}f(H_0+xV)\right]=\tau\left[V\frac{d^{j-1}}{dx^{j-1}}f'(H_0+xV)\right]\] (see Corollary \ref{prop:chain}).
The obtained representations for $\tau\left[\frac{d^j}{dx^j}f(H_0+xV)\right]$ are applied in derivation of explicit formulas for the remainders of non-commutative Taylor-type approximations (described below) in Section \ref{sec:proofs}.

Let $R_p(f)\equiv R_{p,H_0,V}(f)$ denote the remainder of the Taylor-type approximation
\begin{align*}
f(H_0+V)-\sum_{j=0}^{p-1}\frac{1}{j!}\frac{d^j}{dx^j}\bigg|_{x=0}f(H_0+xV)
\end{align*}
of the value of $f(H_0+V)$ at the perturbed operator $H_0+V$ by data determined by the initial operator $H_0$.
Let $\mathcal{W}_p(\Reals)$ denote the set of functions $f\in C^p(\Reals)$ such that for
each $j=0,\dots,p$, the derivative $f^{(j)}$ equals the Fourier transform
$\int_\Reals e^{\i t\la}\,d\mu_{f^{(j)}}(\la)$ of a finite Borel measure $\mu_{f^{(j)}}$.
There exist functions $\xi\equiv\xi_{H_0+V,H_0}$ and $\eta\equiv\eta_{H_0,H_0+V}$, called Krein's and Koplienko's spectral shift functions, respectively, such that when $\tau(|V|)<\infty$,
\begin{align}\label{f-la:4}
\tau[R_1(f)]=\int_\Reals f'(t)\xi(t)\,dt,
\end{align} for $f\in\mathcal{W}_1(\Reals)$ \cite{Krein1} (see also \cite{Azamov,Carey,Lifshits,PellerKr}), and when $\tau(|V|^2)<\infty$,
\begin{align}\label{f-la:5}
\tau[R_2(f)]=\int_\Reals f''(t)\eta(t)\,dt,
\end{align} for $f\in\mathcal{W}_2(\Reals)$ \cite{Kop84} (see also \cite{ds,Neidhardt,PellerKo,convexity}).

It was conjectured in \cite{Kop84} that for $V$ in the Schatten $p$-class, $p\geq 3$, and
$\Mcal=\BH$ (the algebra of bounded operators on $\Hcal$), there exists a real Borel measure $\nu_p\equiv\nu_{p,H_0,V}$, with the total variation bounded by $\frac{\tau(|V|^p)}{p!}$, such that
\begin{align}\label{f-la:7}
\tau[R_p(f)]=\int_\Reals f^{(p)}(t)\,d\nu_p(t),\end{align} for bounded
rational functions $f$. A proof of \eqref{f-la:7} was also suggested in
\cite{Kop84}, but, unfortunately, it contained a mistake (see \cite{ds} for details).

It was proved in \cite[Theorem 5.1]{ds} that \eqref{f-la:7} holds for $f\in\mathcal{W}_p(\Reals)$ when $V$ is in the Hilbert-Schmidt class and $\Mcal=\BH$, with $\nu_p$ a real Borel measure whose total variation is bounded by $\frac{\tau(|V|^2)^{p/2}}{p!}$.
It was shown in \cite{ds} and \cite{unbdds} that $\nu_p$ is absolutely continuous for a bounded and unbounded $H_0$, respectively. Moreover, an explicit formula for the density of $\nu_p$, called the spectral shift function of order $p$, was derived in \cite{ds,unbdds} (see, e.g., \eqref{f-la:nup2} of Theorem \ref{prop:finite}). The trace formula \eqref{f-la:7} was also obtained in the case of $\Mcal$ a general semi-finite von Neumann algebra and $p=3$ \cite[Theorem 5.2]{ds}, with $\nu_3$ absolutely continuous when $H_0$ is bounded. When $H_0$ is unbounded, the trace formula \eqref{f-la:7} with an absolutely continuous measure $\nu_3$ was established in \cite{unbdds} for a set of functions $f$ disjoint from the one assured by the part of \cite[Theorem 5.2]{ds} for an unbounded $H_0$ (this discrepancy is explained in Remark \ref{rem:unbdd_det}).

The proof of existence of the measure $\nu_p$ in \cite{ds} relied on iterated operator integration techniques, while the proofs of the absolute continuity of $\nu_p$ in \cite{ds,unbdds} on analytic function theory techniques. By utilizing the results on operator derivatives and divided differences of Sections \ref{sec:mult} and  \ref{sec:dd}, respectively, we obtain a simple proof of positivity of $\nu_2$ (see Section \ref{sec:mult}), a more direct, unified, proof of the established trace formula \eqref{f-la:7} and the absolute continuity of $\nu_p$ (see Section \ref{sec:proofs}). We also obtain a new representation for the density of $\nu_p$ (see \eqref{f-la:nup1} of Theorem \ref{prop:finite}) and, in the case of a general $\Mcal$ and unbounded $H_0$, extend \eqref{f-la:7} with an absolutely continuous measure $\nu_3$ to a larger (as compared to \cite{unbdds}) set of functions $f$ (see Theorem \ref{prop:infinite}). The ``spectral shift" meaning of the density of $\nu_p$ is demonstrated on an example of commuting operators in a finite von Neumann algebra in Section \ref{sec:proofs}.

\section{Divided differences and splines}

\label{sec:dd}

In this section we collect facts on divided differences and splines to be used in the sequel.

\begin{defi}\label{prop:dddef}
The divided difference of order $p$ is an operation on functions $f$ of one
(real) variable, which we will usually call $\la$, defined recursively as
follows:
\begin{align*}
&\Delta^{(0)}_{\la_1}(f):=f(\la_1),\\
&\dd{\la_{p+1}}{p}{f}:=
\begin{cases}
\frac{\Delta^{(p-1)}_{\la_1,\dots,\la_{p-1},\la_p}(f)-
\Delta^{(p-1)}_{\la_1,\dots,\la_{p-1},\la_{p+1}}(f)}{\la_p-\la_{p+1}}&
\text{ if } \la_p\neq\la_{p+1}\\[2ex]
\frac{\partial}{\partial t}\big|_{t=\la_p}\Delta^{(p-1)}_{\la_1,\dots,\la_{p-1},t}(f)&
\text{ if } \la_p=\la_{p+1}.
\end{cases}
\end{align*}
\end{defi}

The following facts are well-known.

\begin{prop}
\label{prop:dds}
\begin{enumerate}
\item\label{f-la:sim} $($See \cite[Section 4.7, (a)]{Vore}.$)$\\ $\dd{\la_{p+1}}{p}{f}$
    is symmetric in $\la_1,\la_2,\dots,\la_{p+1}$.

\item\label{f-la:dd3}
$($See \cite[Section 4.7]{Vore}.$)$
For $f$ a sufficiently smooth function,
\[\dd{\la_{p+1}}{p}{f}=\sum_{i\in
\mathcal{I}}\sum_{j=0}^{m(\la_i)-1}c_{ij}(\la_1,\dots,\la_{p+1})
f^{(j)}(\la_i).\] Here $\mathcal{I}$ is
the set of indices $i$ for which $\la_i$ are distinct, $m(\la_i)$
is the multiplicity of $\la_i$, and $c_{ij}(\la_1,\dots,\la_{p+1})\in\Complex$.

In particular, if all points $\la_1,\dots,\la_{p+1}$ are distinct, then
\[\dd{\la_{p+1}}{p}{f}=\sum_{j=1}^{p+1}\frac{f(\la_j)}{\prod_{k\in\{1,\dots,p+1\}\setminus\{j\}}
(\la_j-\la_k)}.\]

\item\label{f-la:dd2}
$($See \cite[Section 4.7]{Vore}.$)$\\
$\dd{\la_{p+1}}{p}{a_p\la^p+a_{p-1}\la^{p-1}+\dots +a_1\la+a_0}=a_p$, where
$a_0,a_1,\dots,a_p\in\Complex$.


\item \label{f-la:dd1} $($See \cite[Theorem 6.2 and Theorem 6.3]{Vore}.$)$
For $f\in C^p[a,b]$, the function
\[[a,b]^{p+1}\ni(\la_1,\dots,\la_{p+1})\mapsto\dd{\la_{p+1}}{p}{f}\] is continuous.

\end{enumerate}
\end{prop}

We will need a more specific version of Proposition \ref{prop:dds} \eqref{f-la:dd3}.

\begin{lemma}\label{prop:dd1}
Let $f\in C^1[a,b]$ and $\la_1,\dots,\la_p$ be distinct points in $[a,b]$. Then for any $i\in\{1,\dots,p\}$,
\begin{align*}
&\De_{\la_1,\dots,\la_p,\la_i}^{(p)}{(f)}=
\frac{f'(\la_i)}{\prod_{k\in\{1,\dots,p\}\setminus\{i\}}(\la_i-\la_k)}\\
&\quad+\sum_{j\in\{1,\dots,p\}\setminus\{i\}}\frac{1}{(\la_i-\la_j)^2}
\left(\frac{f(\la_j)}{\prod_{k\in\{1,\dots,p\}\setminus\{i,j\}}(\la_j-\la_k)}-
\frac{f(\la_i)}{\prod_{k\in\{1,\dots,p\}\setminus\{i,j\}}(\la_i-\la_k)}\right)
\end{align*}
\end{lemma}

\begin{proof}Without loss of generality, we may assume that $\la_i=\la_p$. By Proposition \ref{prop:dds} \eqref{f-la:dd3},
\begin{align*}
\De_{\la_1,\dots,\la_{p-1},s}^{(p-1)}{(f)}=
\sum_{j=1}^{p-1}\frac{f(\la_j)}{\prod_{k\in\{1,\dots,p-1\}\setminus\{j\}}
(\la_j-\la_k)(\la_j-s)}+\frac{f(s)}{\prod_{k\in\{1,\dots,p-1\}}(s-\la_k)}.
\end{align*} Next,
\begin{align*}
&\De_{\la_1,\dots,\la_{p-1},\la_p,\la_p}^{(p)}{(f)}=\frac{\partial}{\partial s} \left(\De_{\la_1,\dots,\la_{p-1},s}^{(p-1)}{(f)}\right)\bigg|_{s=\la_p}\\
&\quad=\sum_{j=1}^{p-1}\frac{f(\la_j)}{\prod_{k\in\{1,\dots,p-1\}\setminus\{j\}}
(\la_j-\la_k)(\la_p-\la_j)^2}\\
&\quad\quad+\frac{f'(\la_p)}{\prod_{k\in\{1,\dots,p-1\}}(\la_p-\la_k)}
-f(\la_p)\sum_{j=1}^{p-1}\frac{1}{\prod_{k\in\{1,\dots,p-1\}\setminus\{j\}}
(\la_p-\la_k)(\la_p-\la_j)^2},
\end{align*} which coincides (upon regrouping the terms) with the expression
in the statement of the lemma.
\end{proof}

In the case of repeated knots, the order of the divided difference can be reduced, as it is done in the next lemma.

\begin{lemma}\label{prop:dd2} Let $f\in C^p[a,b]$ and $\la_1,\dots,\la_p\in [a,b]$. Then,
\[\sum_{i=1}^p\De_{\la_1,\dots,\la_p,\la_i}^{(p)}(f)=\dd{\la_p}{p-1}{f'}.\]
\end{lemma}

\begin{proof}
In view of Proposition \ref{prop:dds} \eqref{f-la:dd1}, it is enough to prove the lemma only in the case when all $\la_1,\dots,\la_p$ are distinct. Applying Lemma \ref{prop:dd1} \eqref{f-la:dd1} ensures
\begin{align}
\label{f-la:dd2.1}
&\sum_{i=1}^p\De_{\la_1,\dots,\la_p,\la_i}^{(p)}(f)=
\sum_{i=1}^p\frac{f'(\la_i)}{\prod_{k\in\{1,\dots,p\}\setminus\{i\}}(\la_i-\la_k)}\\
\nonumber
&\quad+\sum_{i=1}^p\sum_{j\in\{1,\dots,p\}\setminus\{i\}}\frac{1}{(\la_i-\la_j)^2}
\left(\frac{f(\la_j)}{\prod_{k\in\{1,\dots,p\}\setminus\{i,j\}}(\la_j-\la_k)}-
\frac{f(\la_i)}{\prod_{k\in\{1,\dots,p\}\setminus\{i,j\}}(\la_i-\la_k)}\right)
\end{align}
By Proposition \ref{prop:dds} \eqref{f-la:dd3}, the first summand in \eqref{f-la:dd2.1} equals $\dd{\la_p}{p-1}{f'}$. The second summand in \eqref{f-la:dd2.1} with double summation sign equals zero; to see it, we group and cancel the terms with indices $(i,j)=(i_1,i_2)$ and $(i,j)=(i_2,i_1)$, where $i_1\neq i_2\in\{1,\dots,p\}$.
\end{proof}

\begin{remark}
Depending on the number of repeated knots of the divided difference in Proposition \ref{prop:dds} \eqref{f-la:dd1} and, subsequently, in Lemma \ref{prop:dd2}, the smoothness assumption on $f$ can be relaxed; see for details \cite[Theorem 6.2 and Theorem 6.3]{Vore}.
\end{remark}

The divided difference of a function in $\mathcal{W}_p(\Reals)$ admits a useful representation
as an integral of products of exponentials, each depending on only one knot of the divided difference.

\begin{prop}\label{prop:dd_wp}$($See \cite[Lemma 2.3]{Azamov0}.$)$
For $f\in\mathcal{W}_p(\Reals)$,
\begin{align*}
\dd{\la_{p+1}}{p}{f}=\int_{\Pi^{(p)}}e^{\i
(s_0-s_1)\la_1}\dots e^{\i (s_{p-1}-s_p)\la_p}e^{\i s_p \la_{p+1}}\,d\sigma_f^{(p)}(s_0,\dots,s_p).
\end{align*} Here
\[{\Pi^{(p)}}=\{(s_0,s_1,\dots,s_p)\in\Reals^{p+1}\st |s_p|\leq\dots\leq|s_1|\leq|s_0|,
\text{ \rm sign}(s_0)=\dots=\text{\rm sign}(s_p)\}\] and
$d\sigma_f^{(p)}(s_0,s_1,\dots,s_p)=\i^p\mu_f(ds_0)ds_1\dots ds_p$, where $f(t)=\frac{1}{\sqrt{2\pi}}\int_\Reals e^{\i t\la}\,d\mu_f(\la)$.
\end{prop}

Below, we list properties of piecewise polynomials, which will appear in representations for the higher order spectral shift functions, and include a representation of the divided difference in terms of its Peano kernel.

\begin{prop}\label{prop:sps}
\begin{enumerate}
\item\label{f-la:dd5} $($See \cite[Section 5.2, (2.3) and (2.6)]{Vore}.$)$\\
The basic spline with the break points $\la_1,\dots,\la_{p+1}$, where at
least two of the values are distinct, is defined by
\[t\mapsto\begin{cases}
\frac{1}{|\la_2-\la_1|}\chi_{(\min\{\la_1,\la_2\},\max\{\la_1,\la_2\})}(t)&
\text{ if } p=1\\[2ex]
\dd{\la_{p+1}}{p}{(\la-t)^{p-1}_+}& \text{ if } p>1\end{cases}.\] Here the
truncated power is defined by $x_+^k=\begin{cases}x^k &\text{ if }x\geq
0\\0&\text{ if }x<0,\end{cases}$ for $k\in\Nats$.

The basic spline is non-negative, supported in
\[[\min\{\la_1,\dots,\la_{p+1}\},\max\{\la_1,\dots,\la_{p+1}\}]\] and
integrable with the integral equal to $1/p$.  (Often the basic spline is
normalized so that its integral equals 1.)

\item\label{f-la:dd4} $($See \cite[Section 5.2, (2.2) and Section 4.7, (c)]{Vore}.$)$

Let $[a,b]\supseteq[\min\{\la_1,\dots,\la_{p+1}\},\max\{\la_1,\dots,\la_{p+1}\}]$.
For $f\in C^p[a,b]$,
\begin{align}\nonumber&\dd{\la_{p+1}}{p}{f}\\&\quad=\label{f-la:dd*}
\begin{cases}
\frac{1}{(p-1)!}\int_a^b
f^{(p)}(t)\dd{\la_{p+1}}{p}{(\la-t)^{p-1}_+}\,dt&\;\text{ if
}\;\exists i_1,i_2\text{ such that } \la_{i_1}\neq \la_{i_2}\\[2ex]
\frac{1}{p!}f^{(p)}(\la_1)&\;\text{ if
}\;\la_1=\la_2=\cdots=\la_{p+1}.
\end{cases}\end{align} The first equality in \eqref{f-la:dd*} also holds for $f\in C^{p-1}[a,b]$, with $f^{(p-1)}$ absolutely continuous and $f^{(p)}$ integrable on $[a,b]$.
\end{enumerate}
\end{prop}

Properties of an antiderivative of the basic spline are written below.

\begin{prop}\label{prop:03}$($See \cite[Lemma 3.1]{unbdds}.$)$
(i) If $\la_1=\dots=\la_{p+1}\in\Reals$, with $p\geq 0$, then
\begin{align}\label{f-la:03.-2}
\dd{\la_{p+1}}{p}{(\la-t)_+^p}=\chi_{(-\infty,\la_1)}(t).
\end{align}
(ii) If not all $\la_1,\dots,\la_{p+1}\in\Reals$ coincide, let $\mathcal{I}$ be
the set of indices $i$ for which $\la_i$ are distinct and let $m(\la_i)$ be
the multiplicity of $\la_i$. Assume that $p\geq 1$ and $M=\max_{i\in\mathcal{I}}m(\la_i)\leq p$. Then, $t\mapsto\dd{\la_{p+1}}{p}{(\la-t)_+^p}\in C^{p-M}(\Reals)$ and
\begin{align}\label{f-la:03.-1}
\dd{\la_{p+1}}{p}{(\la-t)_+^p}=p\int_t^\infty\dd{\la_{p+1}}{p}{(\la-s)_+^{p-1}}\,ds.
\end{align}
\end{prop}

\begin{prop}\label{prop:02}$($See \cite[Lemma 3.2]{unbdds}.$)$
Let $(\la_1,\dots,\la_{p+1})\in\Reals^{p+1}$. Then the function $\dd{\la_{p+1}}{p}{(\la-t)_+^p}$ is decreasing; it is equal to $1$ when $t<\min_{1\leq k\leq p+1}\la_k$ and equal to $0$ when
$t\geq\max_{1\leq k\leq p+1}\la_k$.
\end{prop}

We will need a representation of the divided difference in terms of an antiderivative of the corresponding basic spline \eqref{f-la:03.-1}.

\begin{lemma}\label{prop:dd_my_spline}
Let $[a,b]\supseteq[\min\{\la_1,\dots,\la_{p+1}\},\max\{\la_1,\dots,\la_{p+1}\}]$.
For $f\in C^{p+1}[a,b]$,
\begin{align}
\label{f-la:finite4}
\dd{\la_{p+1}}{p}{f}=\frac{1}{p!}f^{(p)}(a)+\frac{1}{p!}\int_a^b f^{(p+1)}(t)\dd{\la_{p+1}}{p}{(\la-t)^{p}_+}\,dt.
\end{align}
\end{lemma}

\begin{proof}
Assume first that not all $\la_1,\dots,\la_{p+1}$ coincide.
Applying Proposition \ref{prop:sps} \eqref{f-la:dd4} and then integrating by parts and applying the representation \eqref{f-la:03.-1} of Proposition \ref{prop:03} provide
\begin{align*}
\nonumber
&\dd{\la_{p+1}}{p}{f}\\
&\quad=
\frac{1}{(p-1)!}\int_a^b f^{(p)}(t)\dd{\la_{p+1}}{p}{(\la-t)^{p-1}_+}\,dt\\
\nonumber
&\quad=-\frac{1}{p!}\left(f^{(p)}(t)\dd{\la_{p+1}}{p}{(\la-t)^p_+}\right)\bigg|_a^b+
\frac{1}{p!}\int_a^b f^{(p+1)}(t)\dd{\la_{p+1}}{p}{(\la-t)^p_+}\,dt.
\end{align*}
By Proposition \ref{prop:02}, the latter reduces to \eqref{f-la:finite4}.
If $\la_1=\dots=\la_{p+1}$, then by Proposition \ref{prop:sps} \eqref{f-la:dd4},
\begin{align*}
\De_{\la_1,\dots,\la_1}^{(p)}(f)&=\frac{1}{p!}f^{(p)}(\la_1)=
\frac{1}{p!}f^{(p)}(a)+\frac{1}{p!}\int_a^{\la_1}f^{(p+1)}(t)\,dt.
\end{align*}
With use of the representation \eqref{f-la:03.-2} of Proposition \ref{prop:03}, the latter can be rewritten as \eqref{f-la:finite4}.
\end{proof}

\section{Traces of multiple operator integrals}
\label{sec:mult}

In this section, we represent traces of certain multiple operator integrals as multiple scalar integrals. In particular, we obtain useful formulas for the traces of the G\^ateaux derivatives $\tau\left[\frac{d^p}{dx^p}f(H_0+xV)\right]$, where $V=V^*$ is a Hilbert-Schmidt perturbation of a self-adjoint operator $H_0$ and $f\in\mathcal{W}_p(\Reals)$.

\subsection{Multiple spectral measures}
We will need the facts that certain multi-measures extend to finite countably additive
measures.

\begin{prop}\label{prop:m}
Let $2\leq p\in\Nats$ and let $E_1,E_2,\ldots,E_p$ be projection-valued Borel
measures from $\Reals$ to $\Mcal$. Suppose that $V_1,\ldots,V_p$ belong to
$\ncs{2}$. Assume that either $\Mcal=\BH$ or $p=2$. Then there is a unique
(complex) Borel measure $m$ on $\Reals^p$ with total variation not exceeding
the product $\norm{V_1}_2\norm{V_2}_2\cdots\,\norm{V_p}_2$, whose value on
rectangles is given by
\[m(A_1\times A_2\times\cdots\times A_p)
=\tau\big[E_1(A_1)V_1E_2(A_2)V_2\dots V_{p-1}E_p(A_p)V_p\big]\]
for all Borel subsets $A_1,A_2,\ldots,A_p$ of $\Reals$.
\end{prop}

\begin{remark}
In the case of $\Mcal=\BH$ and $V$ a Hilbert-Schmidt operator, Proposition
\ref{prop:m} was obtained in \cite{Birman96,Pavlov}. For a general $\Mcal$ and $V\in\ncs{2}$,
the set function $m$ is known to be of bounded variation only if $p=2$ (see \cite[Section 4]{ds} for a positive result and a counterexample).
\end{remark}

\begin{prop}$($See \cite[Corollary 4.3]{ds}.$)$\label{prop:m1}
Under the assumptions of Proposition \ref{prop:m}, there is a unique
(complex) Borel measure $m_1$ on $\Reals^p$ with total variation not exceeding
the product $\norm{V_1}_2\norm{V_2}_2\cdots\,\norm{V_p}_2$, whose value on
rectangles is given by
\[m_1(A_1\times A_2\times\cdots\times A_p\times A_{p+1})
=\tau\big[E_1(A_1)V_1E_2(A_2)V_2\dots V_{p-1}E_p(A_p)V_pE_1(A_{p+1})\big]\]
for all Borel subsets $A_1,A_2,\ldots,A_p,A_{p+1}$ of $\Reals$.
\end{prop}

In the sequel, we will work with the set functions
\begin{align*}
m_{p,H_0,V}(A_1\times A_2\times\cdots\times A_p)
=\tau\big[E_{H_0}(A_1)VE_{H_0}(A_2)V\dots VE_{H_0}(A_p)V\big],
\end{align*}
\begin{align*}
m_{p,H_0,V}^{(1)}(A_1\times A_2\times\cdots\times A_{p+1})
=\tau\big[E_{H_0}(A_1)VE_{H_0}(A_2)V\dots VE_{H_0}(A_p)VE_{H_0}(A_{p+1})\big]
\end{align*}
and their countably-additive extensions (when they exist), called multiple
spectral measures. Here $A_j,$ $1\leq j\leq p$, are measurable subsets of $\Reals$, $H_0=H_0^*$ is affiliated with $\Mcal$, $E_{H_0}$ is the spectral measure of $H_0$, and $V=V^*\in\ncs{2}$. Clearly, the measures
$m_{p,H_0,V}$ and $m_{p,H_0,V}^{(1)}$ are particular representatives of the measures $m$ and $m_1$, respectively.

\begin{prop}$($See \cite[Theorem 4.5]{ds}.$)$\label{prop:m_free}
Let $\tau$ be a finite trace normalized by $\tau(I)=1$ and let $H_0=H_0^*$ be affiliated with $\Mcal$ and $V=V^*\in\Mcal$. Assume that $(zI-H_0)^{-1}$ and $V$ are free. Then the set functions $m_{p,H_0,V}$ and $m_{p,H_0,V}^{(1)}$ extend to countably additive measures of bounded variation.
\end{prop}

Upon evaluating a trace, some iterated operator integrals can be written as
Lebesgue integrals with respect to ``multiple spectral measures".

\begin{prop}$($See \cite[Lemma 4.9]{ds}.$)$ \label{prop:i_to_m_g}
Assume the hypothesis of Proposition \ref{prop:m}.
Assume that the spectral measures $E_1,E_2,\dots,E_p$ correspond to
self-adjoint operators $H_0,H_1,\dots,H_p$ affiliated with $\Mcal$,
respectively, and that $V_1,V_2,\dots,V_p\in\ncs{2}$. Let $f_1,f_2,\dots,f_p$
be functions in $C_b(\Reals)$ (continuous bounded). Then
\[\tau[f_1(H_1)V_1 f_2(H_2)V_2\dots f_p(H_p)V_p]=
\int_{\Reals^p}f_1(\la_1)f_2(\la_2)\dots f_p(\la_p)\,dm(\la_1,\la_2,\dots,\la_p),\]
with $m$ as in Proposition \ref{prop:m}.
\end{prop}

\begin{remark}\label{prop:i to m_r}
A completely analogous result with $m$ replaced by $m_{p,H_0,V}$ or $m_{p,H_0,V}^{(1)}$ holds under the hypothesis of Proposition \ref{prop:m} or Proposition \ref{prop:m_free}.
\end{remark}


\subsection{Reduction of traces of multiple operator integrals to scalar integrals}

\begin{defi}(\cite[Definition 4.1]{Azamov0}; see also \cite{PellerMult})\label{moi}
Let $H_k=H_k^*$ and $V_k=V_k^*$, with $k=1,\dots,p+1$, be operators defined in $\Hcal$. Assume that $V_k$, $k=1,\dots,p+1$, are bounded. Let $\phi$ be a function representable in the form \begin{align}
\label{f-la:dec}
\phi(\la_1,\la_2,\dots,\la_p,\la_{p+1})=
\int_S\al_1(\la_1,s)\al_2(\la_2,s)\dots\al_p(\la_p,s)\al_{p+1}(\la_{p+1},s)\,d\sigma(s),
\end{align} where $(S,\sigma)$ is a finite measure space and $\al_1,\dots,\al_{p+1}$ are bounded Borel functions on $\Reals\times S$. Then the multiple operator integral
\[\int_{\Reals^{p+1}}\phi(\la_1,\la_2,\dots,\la_p,\la_{p+1})\,
dE_{H_1}(\la_1)V_1\,dE_{H_2}(\la_2)V_2\dots
dE_{H_p}(\la_p)V_p\,dE_{H_{p+1}}(\la_{p+1})\] is defined as the Bochner integral
\begin{align*}
\int_S\al_1(H_1,s)V_1\al_2(H_2,s)V_2\dots\al_p(H_p,s)V_p\,\al_{p+1}(H_{p+1},s)\,d\sigma(s).
\end{align*}
\end{defi}

When the set functions $m$ and $m_1$ admit extensions to finite countably additive measures,
a trace of a multiple operator integral can be represented as a multiple scalar integral.

\begin{lemma}\label{prop:one}
Let $H_1,\dots,H_{p+1}$ be self-adjoint operators affiliated with $\Mcal$ and $V_1,\dots,V_p$ self-adjoint operators in $\ncs{2}$. Let $E_k=E_{H_k}$, for $k=1,\dots,p+1$, and let $\phi$ be a bounded Borel function admitting the representation \eqref{f-la:dec}. Then, the following representations hold.
\begin{enumerate}
\item \label{f-la:one}
For $m_1$ the measure provided by Proposition \ref{prop:m1} or Proposition \ref{prop:m_free},
\begin{align*}
&\tau\left[\int_{\Reals^{p+1}}\phi(\la_1,\dots,\la_p,\la_{p+1})\,dE_{H_1}(\la_1)V_1\dots
dE_{H_p}(\la_p)V_p\,dE_{H_{p+1}}(\la_{p+1})\right]\\
&\quad=\int_{\Reals^{p+1}}\phi(\la_1,\dots,\la_p,\la_{p+1})\,dm_1(\la_1,\dots,\la_p,\la_{p+1}).
\end{align*}
\item \label{f-la:one'}
In the case $H_{p+1}=H_1$, for $m$ the measure provided by Proposition \ref{prop:m} or Proposition \ref{prop:m_free},
\begin{align*}
&\tau\left[\int_{\Reals^{p+1}}\phi(\la_1,\la_2,\dots,\la_p,\la_{p+1})\,dE_{H_1}(\la_1)V_1
dE_{H_2}(\la_2)V_2\dots dE_{H_p}(\la_p)V_p\,dE_{H_1}(\la_{p+1})\right]\\
&\quad=\int_{\Reals^p}\phi(\la_1,\la_2,\dots,\la_p,\la_1)\,dm(\la_1,\la_2,\dots,\la_p).
\end{align*}

\item \label{f-la:one''}
For $m$ the measure provided by Proposition \ref{prop:m} or Proposition \ref{prop:m_free},
\begin{align*}
&\tau\left[V_{p+1}\int_{\Reals^{p+1}}\phi(\la_1,\dots,\la_p,\la_{p+1})\,dE_{H_1}(\la_1)V_1
\dots dE_{H_p}(\la_p)V_p\,dE_{H_{p+1}}(\la_{p+1})\right]\\
&\quad=\int_{\Reals^p}\phi(\la_1,\dots,\la_p,\la_{p+1})\,dm(\la_1,\dots,
\la_p,\la_{p+1}).
\end{align*}
\end{enumerate}
\end{lemma}

\begin{proof}
(1) By \cite[Lemma 3.10 and Remark 4.2]{Azamov0},
\begin{align*}
&\tau\left[\int_S\al_1(H_1,s)V_1\dots \al_p(H_p,s)V_p\,\al_{p+1}(H_{p+1},s)\,d\sigma(s)\right]\\
&\quad=\int_S\tau\left[\al_1(H_1,s)V_1\dots \al_p(H_p,s)V_p\,\al_{p+1}(H_{p+1},s)\right]\,d\sigma(s).
\end{align*}
By Remark \ref{prop:i to m_r}, the latter integral equals
\begin{align*}
\int_S\int_{\Reals^{p+1}}\al_1(\la_1,s)\dots
\al_p(\la_p,s)\al_{p+1}(\la_{p+1},s)\,dm_1(\la_1,\dots,\la_p,\la_{p+1})\,d\sigma(s),
\end{align*} which by Fubini's theorem converts to
\begin{align*}
&\int_{\Reals^{p+1}}\int_S\al_1(\la_1,s)\dots \al_p(\la_p,s)\al_{p+1}(\la_{p+1},s)\,d\sigma(s)\,dm_1(\la_1,\dots,\la_p,\la_{p+1})\\
&\quad=\int_{\Reals^{p+1}}\phi(\la_1,\dots,\la_p,\la_{p+1})\,dm_1(\la_1,\dots,\la_p,\la_{p+1}).
\end{align*}
(2) By \cite[Lemma 3.10 and Remark 4.2]{Azamov0} and cyclicity of the trace,
\begin{align*}
&\tau\left[\int_S\al_1(H_1,s)V_1\al_2(H_2,s)\dots V_p\,\al_{p+1}(H_{p+1},s)\,d\sigma(s)\right]\\
&\quad=\int_S\tau\left[\al_{p+1}(H_1,s)\al_1(H_1,s)V_1\al_2(H_2,s)\dots V_p\right]\,d\sigma(s).
\end{align*}
By Proposition \ref{prop:i_to_m_g}, the latter integral equals
\begin{align*}
\int_S\int_{\Reals^p}\al_{p+1}(\la_1,s)\al_1(\la_1,s)\al_2(\la_2,s)
\dots\al_p(\la_p,s)\,dm(\la_1,\la_2,\dots,\la_p)\,d\sigma(s),
\end{align*}which by Fubini's theorem converts to
\begin{align*}
&\int_{\Reals^p}\int_S\al_1(\la_1,s)\al_2(\la_2,s)
\dots\al_p(\la_p,s)\al_{p+1}(\la_1,s)\,d\sigma(s)\,dm(\la_1,\la_2,\dots,\la_p)\\
&\quad=\int_{\Reals^p}\phi(\la_1,\la_2,\dots,\la_p,\la_1)\,dm(\la_1,\la_2\dots,\la_p).
\end{align*}
(3) By \cite[Lemma 3.7]{Azamov0}
\begin{align*}
&\tau\left[V_{p+1}\int_{\Reals^{p+1}}\phi(\la_1,\dots,\la_p,\la_{p+1})\,dE_{H_1}(\la_1)V_1
\dots dE_{H_p}(\la_p)V_p\,dE_{H_{p+1}}(\la_{p+1})\right]\\
&\quad=\tau\left[\int_S V_{p+1}\al_1(H_1,s)V_1\dots\al_p(H_p,s)V_p\al_{p+1}(H_{p+1},s)\,d\sigma(s)\right].
\end{align*} By following the lines of the proof of \eqref{f-la:one}, we obtain that the latter equals
\begin{align*}
&\int_S\tau\left[V_{p+1}\al_1(H_1,s)V_1\dots\al_p(H_p,s)V_p\al_{p+1}(H_{p+1},s)\right]\,d\sigma(s)\\
&\quad=\int_{\Reals^p}\phi(\la_1,\dots,\la_p,\la_{p+1})\,dm(\la_1,\dots,
\la_p,\la_{p+1}).
\end{align*}
\end{proof}

\begin{remark}
If in the statement of Lemma \ref{prop:one} \eqref{f-la:one'} we change the assumption $V_1,\dots,V_p\in\ncs{2}$ to $V_1,\dots,V_p\in\ncs{p}$, then we obtain
\begin{align*}
&\tau\left[\int_{\Reals^{p+1}}\phi(\la_1,\la_2,\dots,\la_p,\la_{p+1})\,dE_{H_1}(\la_1)V_1
dE_{H_2}(\la_2)V_2\dots dE_{H_p}(\la_p)V_p\,dE_{H_1}(\la_{p+1})\right]\\
&\quad=\tau\left[\int_{\Reals^p}\phi(\la_1,\la_1,\la_2,\dots,\la_p)\,dE_{H_1}(\la_1)V_1
dE_{H_2}(\la_2)V_2\dots dE_{H_p}(\la_p)V_p\right].
\end{align*}
\end{remark}

We have the following representation for the derivative $\frac{d^p}{dx^p}f(H_0+xV)$.

\begin{prop}\label{prop:der}(\cite[Theorem 5.6]{PellerMult}; see also \cite[Theorem 5.7]{Azamov0})
Let $H_0=H_0^*$ be an operator affiliated with $\Mcal$ and $V=V^*$ an operator in $\Mcal$. Then for every $f$ in $\mathcal{W}_p(\Reals)$,
\[\frac{d^p}{dx^p}\bigg|_{x=0}f(H_0+xV)=p!\int_{\Reals^{p+1}}\dd{\la_{p+1}}{p}{f}\,
dE_{H_0}(\la_1)V\dots V\,dE_{H_0}(\la_{p+1}).\]
\end{prop}

The main assumptions of the following results are collected in the format of a hypothesis.

\begin{hyp}\label{hyp}Let $H_0=H_0^*$ be affiliated with $\Mcal$ and $V=V^*\in\ncs{2}$.
Assume that one of the following three conditions is satisfied:
\begin{enumerate}
\item \label{one} $\Mcal=\BH$, $p\geq 2$,
\item \label{two} $2\leq p\leq 3$,
\item \label{three} $\Mcal$ is finite, $p\geq 2$, and $(zI-H_0)^{-1}$ and $V$ are free in $(\Mcal,\tau)$.
\end{enumerate}
\end{hyp}

In the multiple operator integral representation for the derivative $\frac{d^p}{dx^p}f(H_0+xV)$ provided by Proposition \ref{prop:der}, the order of the divided difference can be reduced upon evaluating the trace.

\begin{thm}\label{prop:two}
Assume Hypothesis \ref{hyp}. Then for $f\in\mathcal{W}_p(\Reals)$,
\begin{align}
\label{f-la:two}
\tau\left[\frac{d^p}{dx^p}\bigg|_{x=0}f(H_0+xV)\right]&=
p!\int_{\Reals^{p+1}}\dd{\la_{p+1}}{p}{f}\,dm_{p,H_0,V}^{(1)}(\la_1,\dots,\la_{p+1})\\
\label{f-la:two'}
&=(p-1)!\int_{\Reals^p}\De_{\la_1,\dots,\la_p}^{(p-1)}(f')\,
dm_{p,H_0,V}(\la_1,\dots,\la_p).
\end{align}
\end{thm}

\begin{proof}
By Proposition \ref{prop:dd_wp}, the function $\phi(\la_1,\dots,\la_{p+1})=\dd{\la_{p+1}}{p}{f}$ admits the representation \eqref{f-la:dec},
where $\al_1(\la_1,s)=e^{\i(s_0-s_1)\la_1}$, ..., $\al_p(\la_p,s)=e^{\i(s_{p-1}-s_p)\la_p}$, and $\al_{p+1}(\la_{p+1},s)=e^{\i s_p\la_{p+1}}$. It follows from Proposition \ref{prop:der} and Lemma \ref{prop:one} that
\begin{align}\nonumber
&\tau\left[\frac{d^p}{dx^p}\bigg|_{x=0}f(H_0+xV)\right]\\
\nonumber
&\quad=\tau\left[\int_{\Reals^{p+1}}\Delta_{\la_1,\la_2,\dots,\la_p,\la_{p+1}}^{(p)}(f)
dE_{H_0}(\la_1)VdE_{H_0}(\la_2)V\dots VdE_{H_0}(\la_{p+1})\right]\\
\label{f-la:twopr}
&\quad=\int_{\Reals^p}\Delta_{\la_1,\la_2,\dots,\la_p,\la_{p+1}}^{(p)}(f)\,
dm_{p,H_0,V}^{(1)}(\la_1,\la_2,\dots,\la_p,\la_{p+1})\\
\label{f-la:twopr'}
&\quad=\int_{\Reals^p}\Delta_{\la_1,\la_2,\dots,\la_p,\la_1}^{(p)}(f)\,
dm_{p,H_0,V}(\la_1,\la_2,\dots,\la_p).
\end{align}
Proposition \ref{prop:der} and the representation \eqref{f-la:twopr} imply \eqref{f-la:two}.

To prove that the expressions in \eqref{f-la:two} and \eqref{f-la:two'} are equal, we note first that a trivial renumbering of the variables of integration gives
\begin{align}
\nonumber
&\int_{\Reals^p}\Delta_{\la_1,\la_2,\dots,\la_p,\la_1}^{(p)}(f)\,
dm_{p,H_0,V}(\la_1,\la_2,\dots,\la_p)\\
\label{f-la:two1}
&\quad=\int_{\Reals^p}\Delta_{\la_i,\la_{i+1},\dots,\la_p,\la_1,\dots,\la_{i-1},\la_i}^{(p)}(f)\,
dm_{p,H_0,V}(\la_i,\la_{i+1},\dots,\la_p,\la_1,\dots,\la_{i-1}).
\end{align}
Cyclicity of the trace $\tau$ ensures cyclicity of the measure $m_{p,H_0,V}$, that is,
\begin{align}
\label{f-la:two2}
dm_{p,H_0,V}(\la_i,\la_{i+1},\dots,\la_p,\la_1,\dots,\la_{i-1})=
dm_{p,H_0,V}(\la_1,\dots,\la_{i-1},\la_i,\la_{i+1},\dots,\la_p).
\end{align}
Symmetry of the divided difference (see Proposition \ref{prop:dds} \eqref{f-la:sim}) along with \eqref{f-la:two1} and \eqref{f-la:two2} ensures the equality
\begin{align}
\nonumber
&\int_{\Reals^p}\Delta_{\la_1,\la_2,\dots,\la_p,\la_1}^{(p)}(f)\,
dm_{p,H_0,V}(\la_1,\la_2,\dots,\la_p)\\
\label{f-la:two3}
&\quad=\int_{\Reals^p}\Delta_{\la_1,\la_2,\dots,\la_p,\la_i}^{(p)}(f)\,
dm_{p,H_0,V}(\la_1,\la_2,\dots,\la_p).
\end{align}
It follows from \eqref{f-la:two3} and Lemma \ref{prop:dd2} that
\begin{align}
\nonumber
&p\int_{\Reals^p}\Delta_{\la_1,\la_2,\dots,\la_p,\la_1}^{(p)}(f)\,
dm_{p,H_0,V}(\la_1,\la_2,\dots,\la_p)\\
\nonumber
&\quad=\sum_{i=1}^p\int_{\Reals^p}\Delta_{\la_1,\la_2,\dots,\la_p,\la_i}^{(p)}(f)\,
dm_{p,H_0,V}(\la_1,\la_2,\dots,\la_p)\\
\label{f-la:two4}
&\quad=\int_{\Reals^p}\De_{\la_1,\dots,\la_p}^{(p-1)}(f')\,
dm_{p,H_0,V}(\la_1,\dots,\la_p).
\end{align}
Combination of \eqref{f-la:twopr'} and \eqref{f-la:two4} completes the proof of the theorem.
\end{proof}

As an application of Theorem \ref{prop:two}, we obtain positivity of Koplienko's spectral shift function in the von Neumann algebra setting. In the $\BH$ setting, positivity of $\eta=\eta_2$ was obtained in \cite{Kop84} and in the extended setting for $V\in\ncs{1}$ in \cite{convexity}.

\begin{cor}\label{prop:pos}
Let $H_0=H_0^*$ be affiliated with $\Mcal$ and $V=V^*\in\ncs{2}$. Then $\eta_2\geq 0$.
\end{cor}

\begin{proof}
Due to Koplienko's trace formula \eqref{f-la:5}, it is enough to show that
\[\tau\left[f(H_1)-f(H_0)-\frac{d}{dx}\bigg|_{x=0}f(H_0+xV)\right]\geq 0\] for every $f\in C_c^3(\Reals)\subset\mathcal{W}_2(\Reals)$, with $f''\geq 0$.
We have the integral representation
\begin{align}\label{f-la:irr1}
\tau\left[f(H_1)-f(H_0)-\frac{d}{dx}\bigg|_{x=0}f(H_0+xV)\right]
=\int_0^1(1-x)\tau\left[\frac{d^2}{dx^2}f(H_0+xV)\right]dx
\end{align} (see, e.g., \cite[Theorem 11 and Lemma 3.11]{ds}.) Further, by \eqref{f-la:two1} of Theorem \ref{prop:two} with $p=2$, for every $f\in\mathcal{W}_2(\Reals)$,
\begin{align}\label{f-la:irr2}
\tau\left[\frac{d^2}{dx^2}f(H_0+xV)\right]=\int_{\Reals^2}\De^{(1)}_{\la_0,\la_1}(f')\,
dm_{2,H_0+xV,V}(\la_0,\la_1).
\end{align}
It is easy to derive that the measure $m_{2,H_0+xV,V}$ is non-negative, for every $x\in [0,1]$ (see, e.g., \cite[Lemma 4.7]{ds}). If $f''\geq 0$, then $f'$ is increasing and $\De^{(1)}_{\la_0,\la_1}(f')\geq 0$ for all $\la_0,\la_1$. (The latter follows, for instance, from Proposition \ref{prop:sps}). Thus, if $f''\geq 0$, then the expressions in \eqref{f-la:irr2} and \eqref{f-la:irr1} are non-negative, which completes the proof.
\end{proof}

\begin{cor}\label{prop:two_cor}
Assume Hypothesis \ref{hyp}. Assume, in addition, that $H_0$ is bounded. Let $[a,b]$ be a segment containing $\sigma(H_0)\cup\sigma(H_0+V)$. Then for $f\in\mathcal{W}_p(\Reals)\cap C^{p+1}(\Reals)$,
\begin{align}
\nonumber
&\tau\left[\frac{d^p}{dx^p}\bigg|_{x=0}f(H_0+xV)\right]-\tau(V^p)f^{(p)}(a)\\
\label{f-la:finite5}
&\quad=\int_a^b f^{(p+1)}(t)\int_{[a,b]^{p+1}}\dd{\la_{p+1}}{p}{(\la-t)_+^{p}}\,
dm_{p,H_0,V}^{(1)}(\la_1,\dots,\la_{p+1})\,dt\\
\label{f-la:finite5'}
&\quad=\int_a^b f^{(p+1)}(t)\int_{[a,b]^p}\dd{\la_p}{p-1}{(\la-t)_+^{p-1}}\,
dm_{p,H_0,V}(\la_1,\dots,\la_p)\,dt.
\end{align}
\end{cor}

\begin{proof}
First, note that the measure $m_{p,H_0,V}$ is supported in $[a,b]^p$.
Expanding the integrands in \eqref{f-la:two} and \eqref{f-la:two'} according to \eqref{f-la:finite4} of Lemma \ref{prop:dd_my_spline} and then using Fubini's theorem (the functions $f^{(p+1)}(\cdot)$, $\dd{\la_{p+1}}{p}{(\la-\cdot)_+^{p}}$, and $\dd{\la_p}{p-1}{(\la-\cdot)_+^{p-1}}$ are bounded and the measures $m_{p,H_0,V}^{(1)}$ and $m_{p,H_0,V}$ are finite) provide the representations \eqref{f-la:finite5} and
\eqref{f-la:finite5'}. Here we used the fact that $m_{p,H_0,V}^{(1)}(\Reals^{p+1})=m_{p,H_0,V}(\Reals^p)=\tau(V^p)$.
\end{proof}

The order of an operator derivative inside a trace can be decreased by means of increasing the order of a scalar derivative.

\begin{cor}\label{prop:chain}
Assume Hypothesis \ref{hyp}. Then for $f\in\mathcal{W}_p(\Reals)$,
\begin{align*}
\tau\left[\frac{d^p}{dx^p}\bigg|_{x=0}f(H_0+xV)\right]=
\tau\left[V\frac{d^{p-1}}{dx^{p-1}}\bigg|_{x=0}f'(H_0+xV)\right].
\end{align*}
\end{cor}

\begin{proof}
By Lemma \ref{prop:one} \eqref{f-la:one''} and Proposition \ref{prop:der},
\[(p-1)!\int_{\Reals^p}\De_{\la_1,\dots,\la_p}^{(p-1)}(f')\,
dm_{p,H_0,V}(\la_1,\dots,\la_p)=\tau\left[V\frac{d^{p-1}}{dx^{p-1}}\bigg|_{x=0}f'(H_0+xV)\right],\]
which along with Theorem \ref{prop:two} completes the proof.
\end{proof}

\begin{remark}
The assertions of Corollaries \ref{prop:two_cor} and \ref{prop:chain} remain true if $p=1$, provided $V\in\ncs{1}$, and if $p=2$.
\end{remark}

\section{Properties of the spectral shift measure}

\label{sec:proofs}

In this section, we prove existence of the higher order spectral shift functions and derive some of their properties by implementing a multiple operator integral approach.

\begin{thm}\label{prop:finite}
Assume Hypothesis \ref{hyp}. Assume, in addition, that $H_0$ is bounded.
\begin{enumerate}
\item There exists a unique finite real-valued absolutely continuous measure $\nu_p$ such that the trace formula
\begin{equation}
\label{f-la:tr}\tau[R_{p}(f)]=\int_\Reals f^{(p)}(t)d\nu_p(t)
\end{equation}
holds for $f\in\mathcal{W}_p(\Reals)\cup\Rfr$, where $\Rfr$ denotes the set of rational functions on $\Reals$ with nonreal poles.
\item The density of $\nu_p$ is given by the formulas
\begin{align}
\nonumber
\eta_p(t)&=\frac{\tau(V^{p-1})}{(p-1)!}-\nu_{p-1}((-\infty,t))\\
\label{f-la:nup1}
&\quad-\frac{1}{(p-1)!}\int_{\Reals^p}\dd{\la_p}{p-1}{(\la-t)_+^{p-1}}\,
dm_{p-1,H_0,V}^{(1)}(\la_1,\dots,\la_{p})\\
\nonumber
&=\frac{\tau(V^{p-1})}{(p-1)!}-\nu_{p-1}((-\infty,t))\\
\label{f-la:nup2}
&\quad-\frac{1}{(p-1)!}\int_{\Reals^{p-1}}\dd{\la_{p-1}}{p-2}{(\la-t)_+^{p-2}}\,
dm_{p-1,H_0,V}(\la_1,\dots,\la_{p-1}).
\end{align}
\item The measure $\nu_p$ is supported in the convex hull of the set $\sigma(H_0)\cup\sigma(H_0+V)$ and $\nu_p(\Reals)=\frac{\tau(V^p)}{p!}$.
\end{enumerate}
\end{thm}

\begin{remark}
Theorem \ref{prop:finite}, except for the representation \eqref{f-la:nup1}, was originally proved in \cite[Theorem 5.1, Theorem 5.2, and Theorem 5.6]{ds}. We provide a shorter proof.
\end{remark}

\begin{proof}[Proof of Theorem \ref{prop:finite}]
The proof can be accomplished by induction. The result is known to hold for $p=2$ (see \cite{ds,Kop84}). Assume that the theorem holds when $p$ is replaced with $p-1$. Let $[a,b]$ be a segment containing $\sigma(H_0)\cup\sigma(H_0+V)$. Then $\eta_{p-1}$ is supported in $[a,b]$. Clearly,
\begin{align}
\label{f-la:finite1}
\tau\left[R_{p}(f)\right]=\tau\left[R_{p-1}(f)\right]-
\frac{1}{(p-1)!}\tau\left[\frac{d^{p-1}}{dx^{p-1}}\bigg|_{x=0}f(H_0+xV)\right].
\end{align}

Let $f\in\mathcal{W}_p(\Reals)$. By the induction hypothesis and the representation \eqref{f-la:finite5} of Corollary \ref{prop:two_cor}, the  expression in \eqref{f-la:finite1} equals
\begin{align}
\label{f-la:finite2}
\nonumber
&\int_a^b f^{(p-1)}(t)\eta_{p-1}(t)\,dt-\frac{\tau(V^{p-1})}{(p-1)!}f^{(p-1)}(a)\\
&\quad-\frac{1}{(p-1)!}\int_a^b f^{(p)}(t)\int_{[a,b]^p}\dd{\la_p}{p-1}{(\la-t)_+^{p-1}}\,
dm_{p-1,H_0,V}^{(1)}(\la_1,\dots,\la_p)\,dt.
\end{align}
Integrating by parts in the first integral in \eqref{f-la:finite2} gives
\begin{align}
\nonumber
\int_{[a,b]} f^{(p-1)}(t)\eta_{p-1}(t)\,dt&=
\left(f^{(p-1)}(t)\int_a^t\eta_{p-1}(s)\,ds\right)\bigg|_a^b-\int_a^b f^{(p)}(t)\left(\int_a^t\eta_{p-1}(s)\,ds\right)\,dt\\
\label{f-la:finite3}
&=f^{(p-1)}(b)\frac{\tau(V^{p-1})}{(p-1)!}-\int_a^b f^{(p)}(t)\left(\int_a^t\eta_{p-1}(s)\,ds\right)\,dt.
\end{align}
Combining \eqref{f-la:finite1} - \eqref{f-la:finite3} implies
\begin{align*}
&\tau\left[R_{p}(f)\right]=\left(f^{(p-1)}(b)-f^{(p-1)}(a)\right)
\frac{\tau(V^{p-1})}{(p-1)!}-\int_a^b f^{(p)}(t)\int_a^t\eta_{p-1}(s)\,ds\,dt\\
&\quad
-\int_a^b f^{(p)}(t)\frac{1}{(p-1)!}\int_{[a,b]^p}\dd{\la_p}{p-1}{(\la-t)_+^{p-1}}\,
dm_{p-1,H_0,V}^{(1)}(\la_1,\dots,\la_p)\,dt\\
&=\int_a^b f^{(p)}(t)\bigg(\frac{\tau(V^{p-1})}{(p-1)!}-\nu_{p-1}((a,t))\\
&\quad
-\frac{1}{(p-1)!}\int_{\Reals^p}\dd{\la_p}{p-1}{(\la-t)_+^{p-1}}\,
dm_{p-1,H_0,V}^{(1)}(\la_1,\dots,\la_p)\bigg)\,dt,
\end{align*}
from what the trace formula \eqref{f-la:tr} follows for $f\in\mathcal{W}_p(\Reals)$, with
\begin{align*}
\eta_p(t)&=\frac{\tau(V^{p-1})}{(p-1)!}-\nu_{p-1}((a,t))\\
&\quad-\frac{1}{(p-1)!}\int_{\Reals^p}\dd{\la_p}{p-1}{(\la-t)_+^{p-1}}\,
dm_{p-1,H_0,V}^{(1)}(\la_1,\dots,\la_p).
\end{align*}
Let $[c,d]$ denote the convex hull of $\sigma(H_0)\cup\sigma(H_0+V)$.
By the induction hypothesis, $\nu_{p-1}((-\infty,c))=\nu_{p-1}((d,\infty))=0$ and $\nu_{p-1}([c,d])=\frac{\tau(V^{p-1})}{(p-1)!}$. By Proposition \ref{prop:02},
\begin{align*}
&\int_{\Reals^p}\dd{\la_p}{p-1}{(\la-t)_+^{p-1}}\,dm_{p-1,H_0,V}^{(1)}(\la_1,\dots,\la_p)\\
&\quad=
\begin{cases}
m_{p-1,H_0,V}^{(1)}(\Reals^p)=\frac{\tau(V^{p-1})}{(p-1)!} & \text{ if } t<c\\
0 & \text{ if } t>d
\end{cases}.
\end{align*}
Therefore, \eqref{f-la:nup1} holds and the measure $\nu_p$ is supported in $[c,d]$. To extend
the trace formula \eqref{f-la:tr} to $f$ a polynomial, we apply \eqref{f-la:tr} to a function $g\in\mathcal{W}_p(\Reals)$, which coincides with $f$ on a segment containing $\cup_{x\in[-1,1]}\sigma(H_0+xV)$, and get
\begin{align*}
\tau[R_{p}(f)]=\tau[R_{p}(g)]=\int_\Reals g^{(p)}\,d\nu_p(t)=\int_\Reals f^{(p)}(t)\,d\nu_p(t).
\end{align*} To obtain the equality $\nu_p(\Reals)=\frac{\tau(V^p)}{p!}$, we apply \eqref{f-la:tr} to $f(t)=t^p$. The measure $\nu_p$ is finite since it is compactly supported and its density is bounded.

The proof of \eqref{f-la:nup2} is completely analogous to the proof of \eqref{f-la:nup1}, where the only difference consists in applying \eqref{f-la:finite5'} (instead of \eqref{f-la:finite5}) to the second summand in \eqref{f-la:finite1}.
\end{proof}

The techniques used in the proof of Theorem \ref{prop:finite} also work in the case of an unbounded operator $H_0$, provided $f^{(p)}\in L^1(\Reals)$ and $\nu_{p-1}$ is known to be finite.

\begin{thm}\label{prop:infinite}
Let $H_0=H_0^*$ be an operator affiliated with $\Mcal$, $V=V^*$ an operator in $\ncs{2}$ and $p=3$. Then for $f\in C_c^p(\Reals)\cup\Rfr_b$, where $\Rfr_b$ is the subset of bounded functions in $\Rfr$, the representations \eqref{f-la:tr} - \eqref{f-la:nup2} hold.
\end{thm}

\begin{proof}
The proof is very similar to the one of Theorem \ref{prop:finite}, so we provide only a brief sketch. Clearly,
\begin{align}
\label{f-la:infinite3}
0&=\lim_{a\rightarrow -\infty,\,b\rightarrow\infty}\left(f^{(p-1)}(b)-f^{(p-1)}(a)\right)\frac{\tau(V^{p-1})}{(p-1)!}
=\int_\Reals f^{(p)}(t)\frac{\tau(V^{p-1})}{(p-1)!}\,dt.
\end{align}
By letting $a\rightarrow -\infty$ and $b\rightarrow\infty$ in \eqref{f-la:finite3}, we obtain
\begin{align}
\label{f-la:infinite1}
\int_\Reals f^{(p-1)}(t)\eta_{p-1}(t)dt=-\int_\Reals f^{(p)}(t)\left(\int_{-\infty}^t\eta_{p-1}(s)\,ds\right)\,dt.
\end{align}
By letting $a\rightarrow -\infty$ and $b\rightarrow\infty$ in \eqref{f-la:finite4}, we obtain
\begin{align*}
\dd{\la_{p}}{p-1}{f}
=\frac{1}{(p-1)!}\int_\Reals f^{(p)}(t)\dd{\la_p}{p-1}{(\la-t)_+^{p-1}}\,dt
\end{align*} and, subsequently,
\begin{align}
\nonumber
&\tau\left[\frac{d^{p-1}}{dx^{p-1}}\bigg|_{x=0}f(H_0+xV)\right]\\
\label{f-la:infinite2}
&\quad=\int_\Reals f^{(p)}(t)\int_{\Reals^{p}}\dd{\la_{p}}{p-1}{(\la-t)_+^{p-1}}\,
dm_{p-1,H_0,V}^{(1)}(\la_1,\dots,\la_{p})\,dt\\
\label{f-la:infinite2'}
&\quad=\int_\Reals f^{(p)}(t)\int_{\Reals^{p-1}}\dd{\la_{p-1}}{p-2}{(\la-t)_+^{p-2}}\,
dm_{p-1,H_0,V}(\la_1,\dots,\la_{p-1})\,dt
\end{align} (see the proof of Corollary \ref{prop:two_cor}).

We note that the integral $\int_{-\infty}^t\eta_{p-1}(s)\,ds$ is well defined since $\eta_{p-1}=\eta_2$ is integrable (see, e.g.,  discussion in the introductory section of \cite{unbdds}).
Combining \eqref{f-la:infinite3} - \eqref{f-la:infinite2'}, as it was done in the case of a bounded $H_0$, completes the proof.
\end{proof}

\begin{remark}\label{rem:unbdd_det}
The trace formula \eqref{f-la:tr} for $f\in C_c^\infty(\Reals)$ (in fact, $f\in C_c^{p+1}(\Reals)$ also works) was obtained in \cite[Theorem 5.2]{ds}, without establishing the absolute continuity of the measure $\nu_3$ when $H_0$ is unbounded. The trace formula
\begin{align}
\label{f-la:treta3}
\tau\left[R_{3}(f)\right]=\int_\Reals f^{'''}(t)\eta_3(t)\,dt,
\end{align} for an unbounded $H_0$,
with $\eta_3$ given by \eqref{f-la:nup2}, was proved in \cite[Theorem 4.1]{unbdds} only
for $f\in\Rfr_b$. The results of Theorem \ref{prop:two} have allowed to obtain \eqref{f-la:treta3} for both $f\in C_c^p(\Reals)$ and $f\in\Rfr_b$. The same approach proves existence of the spectral shift function of order $p\geq 3$ for an unbounded $H_0$, when $\Mcal=\BH$ (the original proofs in \cite{ds,unbdds} were based on the analysis of the Cauchy transform of the measure $\nu_p$).
A substantial obstacle in establishing \eqref{f-la:tr} for $\tau(|V|^p)<\infty$, with $p>2$ (unless $\tau(|V|^2)<\infty$), is non-extendibility of the set function $m_{p,H_0,V}$ to a finite countably additive measure on $\Reals^p$ (see a counterexample in \cite[Section 4]{ds}). An analogous problem has caused a delay in establishing \eqref{f-la:treta3} in the von Neumann algebra setting; the set function $m_{3,H_0,V}$ can fail to extend to a finite measure even if $\tau$ is finite (and $\dim(\Hcal)=\infty$) \cite[Section 4]{ds}. That is why the approach of \cite{ds,unbdds} working for every Hilbert-Schmidt $V=V^*\in\Mcal=\BH$ was not so successful in the general von Neumann algebra setting.
\end{remark}

Below we provide an example, which demonstrates that the density $\eta_p$, with $p>3$, reflects information about the shift of the spectrum of an operator $H_0$ under a perturbation $V$, similar to the known case of $p=2$.

\medskip

{\bf Example.} Assume that $\tau$ is finite. Let $H_0=H_0^*$ and $V=V^*$ be commuting operators in $\Mcal$. Then we have the trace formula \eqref{f-la:tr} with an absolutely continuous measure $\nu_p\equiv\nu_{p,H_0,V}$, whose density is given by
\begin{align}\label{f-la:eta_comm1}
\eta_p(t)&=\frac{1}{(p-1)!}\tau\left[(H_0+V-tI)^{p-1}\left(E_{H_0}(t)-E_{H_0+V}(t)\right)\right]\\
\label{f-la:eta_comm2}
&=\frac{1}{(p-1)!}\tau\left[(H_0+V-tI)^{p-1}\left(E_{H_0}(t)E_{H_0+V}(t)^\perp-E_{H_0}(t)^\perp E_{H_0+V}(t)\right)\right].
\end{align}
Here $E_{H_0}(t)$ denotes the spectral projection $E_{H_0}((-\infty,t))$. If $H_0$ and $H_0+V$ are two commuting finite dimensional matrices with the eigenvalues $t_k^\circ$ and $t_k$,  respectively, and $\tau$ is the standard trace, then \eqref{f-la:eta_comm2} computes the net sum of signed powers of distances from $t$ to those eigenvalues $t_k$, which happen to be on the opposite side of $t$ with the eigenvalues $t_k^\circ$; the precise formula is
\[\eta_p(t)=\frac{1}{(p-1)!}\sum_{k\in\{k\st(t_k^\circ-t)(t_k-t)\leq
0\}}\big(\text{sign}({t_k-t})\big)(t_k-t)^{p-1}.\]

The representation \eqref{f-la:eta_comm2} follows directly from \eqref{f-la:eta_comm1} (see  \cite[Lemma 2.6]{convexity}). One can prove existence of an absolutely continuous measure $\nu_p$ satisfying \eqref{f-la:tr}, with the density given by \eqref{f-la:eta_comm1}, by induction on $p$. In the case of $p=1$ and $\Mcal$ the algebra of matrices on a finite dimensional Hilbert space, the formula \eqref{f-la:eta_comm1} is well-known and goes back to \cite{Lifshits} and, in the case of a general finite $\Mcal$, it is discussed in \cite{Azamov}.
The formula in the case of $p=2$ is due to \cite[Lemma 5.2]{convexity}.
To prove \eqref{f-la:tr} and \eqref{f-la:eta_comm1} for $p>3$, firstly we note that for $\phi$ representable in the form \eqref{f-la:dec},
\begin{align*}
&\int_{\Reals^{p+1}}\phi(\la_1,\la_2,\dots,\la_p,\la_{p+1})\,
dE_{H_1}(\la_1)V\,dE_{H_2}(\la_2)V\dots
dE_{H_p}(\la_p)V\,dE_{H_{p+1}}(\la_{p+1})\\
&\quad=V^p\int_S\al_1(H_1,s)\al_2(H_2,s)\dots\al_p(H_p,s)\,\al_{p+1}(H_{p+1},s)\,d\sigma(s).
\end{align*} Therefore, the formula \eqref{f-la:two} for the trace of an operator derivative rewrites as
\begin{align*}
\tau\left[\frac{d^p}{dx^p}\bigg|_{x=0}f(H_0+xV)\right]=\int_\Reals f^{(p)}(\la)\,d\tau[V^pE_{H_0}(\la)]=\tau[V^pf^{(p)}(H_0)],
\end{align*} for $f\in\mathcal{W}_p(\Reals)\cup\mathfrak{R}$, and hence,
\begin{align}\label{f-la:binomial0}
\tau[R_{p+1}(f)]=\tau[R_p(f)]-\frac{1}{p!}\tau[V^p f^{(p)}(H_0)].
\end{align}
We suppose that \eqref{f-la:tr} holds with $d\nu_p(t)=\eta_p(t)dt$, where $\eta_p$ is given by \eqref{f-la:eta_comm1}, and derive
\[\tau[R_{p+1}(f)]=\int_\Reals f^{(p+1)}(t)\eta_{p+1}(t)\,dt.\] Let $H=H_0+V$. By the binomial theorem we obtain
\begin{align}\label{f-la:binomial1}\nonumber
\frac{1}{p!}\tau[V^pf^{(p)}(H_0)]&=
\frac{1}{p!}\tau[(H-H_0)^pf^{(p)}(H_0)-(H-H)^pf^{(p)}(H)]\\
&=\frac{1}{p!}\sum_{k=0}^p \begin{pmatrix}p\\k\end{pmatrix}(-1)^k\tau[H^{p-k}H_0^kf^{(p)}(H_0)-H^{p-k}H^kf^{(p)}(H)],
\end{align}
which by the spectral theorem can be written as
\begin{align}\label{f-la:binomial2}\nonumber
&\frac{1}{p!}\sum_{k=0}^p\begin{pmatrix}p\\k\end{pmatrix}(-1)^k\tau\left[H^{p-k}\int_\Reals t^kf^{(p)}(t)\,d\left(E_{H_0}(t)-E_H(t)\right)\right]\\
&\quad=\frac{1}{p!}\sum_{k=0}^p\begin{pmatrix}p\\k\end{pmatrix}(-1)^k\int_\Reals t^kf^{(p)}(t)\,d\tau\left[H^{p-k}\left(E_{H_0}(t)-E_H(t)\right)\right].
\end{align}
Integrating by parts in \eqref{f-la:binomial2} gives
\begin{align}\label{f-la:binomial3}\nonumber
&\frac{1}{p!}\sum_{k=1}^p\begin{pmatrix}p\\k\end{pmatrix}(-1)^{k+1}k\int_\Reals t^{k-1}f^{(p)}(t)\tau\left[H^{p-k}\left(E_{H_0}(t)-E_H(t)\right)\right]\,dt\\
&\quad+\frac{1}{p!}\sum_{k=0}^p\begin{pmatrix}p\\k\end{pmatrix}(-1)^{k+1}\int_\Reals t^kf^{(p+1)}(t)\tau\left[H^{p-k}\left(E_{H_0}(t)-E_H(t)\right)\right]\,dt,
\end{align}
which by the binomial theorem can be written as
\begin{align}\label{f-la:binomial4}\nonumber
&\int_\Reals f^{(p)}(t)\frac{1}{(p-1)!}\tau\left[(H-tI)^{p-1}\left(E_{H_0}(t)-E_H(t)\right)\right]\,dt\\
&\quad-\int_\Reals f^{(p+1)}(t)\frac{1}{p!}\,d\tau\left[(H-tI)^p\left(E_{H_0}(t)-E_H(t)\right)\right]\,dt.
\end{align}
By the induction hypothesis, the first summand in \eqref{f-la:binomial4} equals
$\tau[R_p(f)]$. Thus, combining \eqref{f-la:binomial0} - \eqref{f-la:binomial4} completes the proof of \eqref{f-la:eta_comm1}.

\bibliographystyle{plain}

\end{document}